\documentclass{amsart}
\usepackage{amsmath}
\usepackage{amssymb}
\usepackage{graphicx}
\usepackage{latexsym}
\usepackage{color}
\usepackage{amscd}
\usepackage[all]{xy}
\usepackage{enumerate}
\usepackage{enumitem}
\usepackage{hyperref}
\hypersetup{hidelinks}
\usepackage{subfigure}
\usepackage{flafter}
\usepackage{float}
\usepackage{soul}
\usepackage{comment}
\usepackage{epsfig}
\usepackage{epstopdf}
\usepackage{marginnote}
\usepackage{listings}
\usepackage{xcolor}

\usepackage{mparhack}
\usepackage{todonotes}
\usepackage[normalem]{ulem}

\usepackage{tikz}
\usetikzlibrary{arrows.meta,calc,decorations.markings,positioning}

\usepackage{multirow}

\newtheorem{thm}{Theorem}

\newtheorem{theorem}[thm]{Theorem}
\newtheorem{lemma}[thm]{Lemma}
\newtheorem{corollary}[thm]{Corollary}

\theoremstyle{definition}

\newtheorem*{definition*}{Definition}

\newtheorem{remark}[thm]{Remark}

\newcommand{\red}{\textcolor{red}}

\newcommand{\R}{\mathbb{R}}

\newcommand{\Z}{\mathbb{Z}}

\def \bdry {\partial}

\newcommand{\twprod}{\mathbin{\mathchoice%
    {\ooalign{\raise1.15ex\hbox{$\scriptstyle\sim$}\cr\hidewidth$\times$\hidewidth\cr}}%
    {\ooalign{\raise1.15ex\hbox{$\scriptstyle\sim$}\cr\hidewidth$\times$\hidewidth\cr}}%
    {\ooalign{\raise.85ex\hbox{$\scriptscriptstyle\sim$}\cr\hidewidth$\scriptstyle\times$\hidewidth\cr}}%
    {\ooalign{\raise.65ex\hbox{$\scriptscriptstyle\sim$}\cr\hidewidth$\scriptscriptstyle\times$\hidewidth\cr}}%
    }}

\newcommand{\M}{\rm{Mod}}

\usepackage{scalerel}
\newcommand{\DT}[1]{t_{\scaleto{\mathstrut #1}{9pt}}}

\allowdisplaybreaks[4]

\definecolor{alphacolor}{RGB}{0,91,150}
\definecolor{betacolor}{RGB}{102,83,124}
\definecolor{fixedcolor}{RGB}{0,112,91}
\definecolor{blockcolor}{RGB}{166,32,105}
\definecolor{lineone}{RGB}{28,92,151}
\definecolor{linetwo}{RGB}{178,92,0}
\definecolor{linethree}{RGB}{73,132,57}
\definecolor{linefour}{RGB}{129,73,153}
\definecolor{linefive}{RGB}{160,45,55}
\definecolor{auxblue}{RGB}{36,151,219}

\tikzset{
  pagebdry/.style={draw=black!78,fill=black!5,line width=.72pt,
    rounded corners=8pt},
  hole/.style={draw=black!78,fill=white,line width=.58pt},
  alpha curve/.style={draw=alphacolor,line width=1.15pt,
    line cap=round,line join=round},
  beta curve/.style={draw=betacolor,line width=1.15pt,
    line cap=round,line join=round},
  fixed curve/.style={draw=fixedcolor,line width=1.15pt,
    line cap=round,line join=round},
  block curve/.style={draw=blockcolor,line width=1.15pt,
    line cap=round,line join=round},
  auxiliary blue curve/.style={draw=auxblue,line width=1.15pt,
    line cap=round,line join=round},
  basisstem/.style={draw=black!52,line width=.56pt,
    line cap=round,line join=round},
  basisloop/.style={draw=black!72,line width=.72pt,
    line cap=round,line join=round},
  arrpoint/.style={circle,fill=black,inner sep=1.3pt}
}

\newcommand{\PageSetup}{%
  \coordinate (h1) at (-2,0);
  \coordinate (h2) at (-1,0);
  \coordinate (h3) at (0,0);
  \coordinate (h4) at (1,0);
  \coordinate (h5) at (2,0);
  \draw[pagebdry] (-2.78,-1.10) rectangle (2.78,1.10);
}
\newcommand{\DrawHoles}{%
  \foreach \i in {1,...,5}{\draw[hole] (h\i) circle (.155);}
}

\newcommand{\CurveAlpha}{\draw[alpha curve]
 (-2.38,0) .. controls (-2.38,.47) and (-.62,.47) .. (-.62,0)
 .. controls (-.62,-.47) and (-2.38,-.47) .. cycle;}
\newcommand{\CurveBeta}{\draw[beta curve]
 (-.38,0) .. controls (-.38,.47) and (1.38,.47) .. (1.38,0)
 .. controls (1.38,-.47) and (-.38,-.47) .. cycle;}
\newcommand{\CurveU}{\draw[block curve]
 (-1.38,0) .. controls (-1.38,.47) and (.38,.47) .. (.38,0)
 .. controls (.38,-.47) and (-1.38,-.47) .. cycle;}
\newcommand{\CurveD}{\draw[fixed curve]
  (-2.42,0)
  .. controls (-2.42,.48) and (-2.12,.56) .. (-1.72,.56)
  .. controls (-1.46,.56) and (-1.36,-.34) .. (-1.00,-.34)
  .. controls (-.64,-.34) and (-.54,.56) .. (-.28,.56)
  .. controls (.12,.56) and (.42,.48) .. (.42,0)
  .. controls (.42,-.48) and (.12,-.58) .. (-.28,-.58)
  .. controls (-.65,-.60) and (-1.35,-.60) .. (-1.72,-.58)
  .. controls (-2.12,-.58) and (-2.42,-.48) .. cycle;}
\newcommand{\CurveV}{\draw[block curve]
  (-1.42,0)
  .. controls (-1.42,.45) and (-1.12,.56) .. (-.72,.56)
  .. controls (-.46,.56) and (-.36,-.34) .. (0,-.34)
  .. controls (.36,-.34) and (.46,.56) .. (.72,.56)
  .. controls (1.20,.58) and (2.42,.48) .. (2.42,0)
  .. controls (2.42,-.48) and (1.60,-.58) .. (.72,-.58)
  .. controls (.20,-.58) and (-.20,-.58) .. (-.72,-.58)
  .. controls (-1.12,-.58) and (-1.42,-.45) .. cycle;}
\newcommand{\CurveW}{\draw[block curve]
  (-2.42,0)
  .. controls (-2.42,.50) and (-1.40,.62) .. (0,.62)
  .. controls (1.40,.62) and (2.42,.50) .. (2.42,0)
  .. controls (2.42,-.48) and (2.15,-.56) .. (1.60,-.56)
  .. controls (1.38,-.56) and (1.30,.34) .. (1.00,.34)
  .. controls (.70,.34) and (.62,-.56) .. (.40,-.56)
  .. controls (.18,-.58) and (-.18,-.58) .. (-.40,-.56)
  .. controls (-.62,-.56) and (-.70,.34) .. (-1.00,.34)
  .. controls (-1.30,.34) and (-1.38,-.56) .. (-1.60,-.56)
  .. controls (-2.15,-.56) and (-2.42,-.48) .. cycle;}
\newcommand{\CurveZ}{\draw[block curve]
  (-2.42,0)
  .. controls (-2.42,.48) and (-1.50,.60) .. (-.55,.60)
  .. controls (-.34,.60) and (-.30,-.34) .. (0,-.34)
  .. controls (.30,-.34) and (.34,.60) .. (.55,.60)
  .. controls (.82,.60) and (1.42,.48) .. (1.42,0)
  .. controls (1.42,-.48) and (.35,-.60) .. (-.45,-.60)
  .. controls (-.66,-.60) and (-.70,.34) .. (-1.00,.34)
  .. controls (-1.30,.34) and (-1.34,-.60) .. (-1.55,-.60)
  .. controls (-2.15,-.60) and (-2.42,-.48) .. cycle;}

\newcommand{\CurveAThirteen}{\draw[auxiliary blue curve]
  (-2.42,0)
  .. controls (-2.42,.48) and (-2.15,.72) .. (-1.76,.72)
  .. controls (-1.38,.72) and (-.62,.72) .. (-.24,.72)
  .. controls (.15,.72) and (.42,.48) .. (.42,0)
  .. controls (.42,-.28) and (.18,-.34) .. (-.12,-.34)
  .. controls (-.38,-.34) and (-.38,.52) .. (-.66,.52)
  .. controls (-.88,.52) and (-1.12,.52) .. (-1.34,.52)
  .. controls (-1.62,.52) and (-1.62,-.34) .. (-1.88,-.34)
  .. controls (-2.18,-.34) and (-2.42,-.28) .. cycle;}

\newcommand{\CurveATwentyFour}{\draw[auxiliary blue curve]
  (-1.42,0)
  .. controls (-1.42,.28) and (-1.18,.32) .. (-.92,.32)
  .. controls (-.66,.32) and (-.56,-.58) .. (-.32,-.58)
  .. controls (-.12,-.58) and (.12,-.58) .. (.32,-.58)
  .. controls (.56,-.58) and (.66,.32) .. (.92,.32)
  .. controls (1.18,.32) and (1.42,.28) .. (1.42,0)
  .. controls (1.42,-.48) and (1.16,-.82) .. (.78,-.82)
  .. controls (.52,-.82) and (-.52,-.82) .. (-.78,-.82)
  .. controls (-1.16,-.82) and (-1.42,-.48) .. cycle;}

\newcommand{\CurveGamma}{\draw[auxiliary blue curve]
  (-2.48,0)
  .. controls (-2.48,.76) and (1.48,.76) .. (1.48,0)
  .. controls (1.48,-.76) and (-2.48,-.76) .. cycle;}

\newcommand{\CurveLambda}{\draw[auxiliary blue curve]
  (-2.50,0)
  .. controls (-2.50,.48) and (-2.20,.58) .. (-1.80,.58)
  .. controls (-1.30,.60) and (-.70,.58) .. (-.48,.52)
  .. controls (-.28,.46) and (-.30,-.34) .. (0,-.34)
  .. controls (.30,-.34) and (.28,.46) .. (.48,.52)
  .. controls (.70,.58) and (1.45,.58) .. (1.45,0)
  .. controls (1.45,-.50) and (.90,-.62) .. (.48,-.60)
  .. controls (-.20,-.62) and (-1.20,-.62) .. (-1.80,-.60)
  .. controls (-2.20,-.58) and (-2.50,-.48) .. cycle;}

\newcommand{\CurveMu}{\draw[auxiliary blue curve]
  (-1.48,0)
  .. controls (-1.48,.48) and (1.48,.48) .. (1.48,0)
  .. controls (1.48,-.48) and (-1.48,-.48) .. cycle;}

\newcommand{\DrawBasis}{%
  \coordinate (bp) at (2.48,-.78);
  \foreach \xx/\ii/\yy in {
    -2/1/-.96,
    -1/2/-.89,
     0/3/-.82,
     1/4/-.75,
     2/5/-.68}{%
    \draw[basisstem] (bp)
      .. controls (2.28,\yy) and ({\xx+.20},\yy) .. (\xx,-.28);
    \draw[basisloop] (\xx,0) circle (.27);
    \draw[basisloop,-{Stealth[length=2.3pt,width=2.1pt]}]
      ($ (\xx,0)+(25:.27) $)
      arc[start angle=25,end angle=78,radius=.27];
    \node[font=\tiny] at (\xx,.48) {$x_{\ii}$};
  }%
  \node[font=\scriptsize,inner sep=0pt] at (bp) {$*$};
}

\begin{document}

\title[Planar contact $3$--manifolds with infinitely many Stein fillings]
{Planar contact $3$--manifolds with infinitely many Stein fillings}

\author[R. \.{I}. Baykur]{R. \.{I}nan\c{c} Baykur}
\address{Department of Mathematics and Statistics, University of Massachusetts, Amherst, MA 01003, USA}
\email{inanc.baykur@umass.edu}

\begin{abstract}
We prove that there are infinitely many closed contact $3$--manifolds supported
by planar open books, each admitting infinitely many pairwise non-homeomorphic
Stein fillings. This answers K3 Problem 4.105. As a corollary, there are contact $3$--manifolds that admit infinitely
many Stein fillings but do not admit arbitrarily large ones.
\end{abstract}

\maketitle

\setcounter{secnumdepth}{2}
\setcounter{section}{0}

\section{Introduction}

Planar contact $3$--manifolds, namely closed contact $3$--manifolds
supported by genus-zero open books, are subject to strong constraints on
the topology of their symplectic fillings.  The decisive structural result
is Wendl's theorem: every minimal strong symplectic filling of a contact
$3$--manifold supported by a fixed planar open book is, up to symplectic
deformation, an allowable Lefschetz fibration with the same planar fiber
\cite{Wendl2010}.  Weak symplectic fillings reduce to the same setting after
deformation \cite{NiederkrugerWendl2011}, while positive allowable
Lefschetz fibrations over the disk carry Stein structures
\cite{LoiPiergallini2001,AkbulutOzbagci2001}.  Thus the filling problem reduces to the study of
positive Dehn twist factorizations of a planar mapping class.

The intersection form of every such filling is negative definite
\cite{Etnyre2004}.  For Stein fillings, the second Betti number is uniformly
bounded \cite{Plamenevskaya2012}; further restrictions on minimal
symplectic fillings, including obstructions to certain intersection patterns
and to symplectic surfaces of positive genus, were obtained in
\cite{GhigginiGollaPlamenevskaya2020}.  The Euler characteristic and
signature are bounded \cite{Wand2012,Kaloti2015,LisiWendl2021}, and for a
fixed planar monodromy only finitely many integral homology groups can occur
\cite{BaykurMondenVHM2017,BaykurVHM2018}.  Moreover, many standard planar
families, including lens spaces and links of certain rational surface
singularities, have at most finitely many fillings; see e.g.
\cite{Lisca2008,PlamenevskayaVHM2010,BhupalOzbagci2016,
Plamenevskaya2022,Schonenberger2007,EtnyreRoy2021}.

Indeed, it has been conjectured that every planar contact
$3$--manifold should admit only finitely many deformation classes of minimal symplectic fillings \cite{LisiWendl2021}; this question also constitutes the $K3$ problem  \cite[Problem~4.105]{K3ProblemList}.  We disprove this conjecture with the following result.

\begin{theorem}\label{thm:main}
There are infinitely many planar contact $3$--manifolds, each admitting
infinitely many pairwise non-homeomorphic Stein fillings.  Our simplest example
is supported by a planar open book whose page has six boundary components.  Its infinitely many Stein
fillings $X_n$, $n\geq 0$, satisfy
\[
 H_0(X_n;\Z)=H_1(X_n;\Z)=H_2(X_n;\Z)=\Z,
 \qquad
 \pi_1(X_n)\cong\Z/(2n+1)\rtimes_{-1}\Z.
\]
\end{theorem}

\noindent Since Stein fillings are symplectically minimal, the
pairwise non-homeomorphic fillings in Theorem~\ref{thm:main} also give pairwise distinct deformation classes of minimal symplectic fillings.

Many contact $3$--manifolds, including all tight contact structures on lens
spaces, admit only finitely many Stein fillings up to diffeomorphism.  At the other extreme,
many contact $3$--manifolds supported by higher genus open books admit
arbitrarily large Stein fillings, that is, fillings with unbounded Euler
characteristic \cite{BaykurVHM2015,BaykurVHM2016}.  Together with the uniform boundedness of the planar fillings, Theorem~\ref{thm:main} shows that the second inclusion below is also strict, and reveals the existence of a new class of contact $3$--manifolds:  
\begin{equation*}
 \{\text{Stein fillable}\}
 \supsetneq
 \{\text{infinitely Stein fillable}\}
 \, \red{\supsetneq} \,
 \{\text{arbitrarily large Stein fillable}\}.
\end{equation*}

\begin{corollary}\label{cor:bounded-infinite}
There are infinitely many contact $3$--manifolds that admit infinitely many
pairwise non-homeomorphic Stein fillings but do not admit arbitrarily large
Stein fillings.
\end{corollary}

\section{Proofs}

Let $\Sigma=\Sigma_0^6$ be a disk with an outer boundary $\delta_0$ and five inner boundary components
$\delta_1,\ldots,\delta_5$, ordered from left to right in the figures below, and set $
 \Gamma_0^6:=\M(\Sigma,\bdry\Sigma)$ for shorthand.
Products in this article are written in functional order, and we use $\psi\DT{c}\psi^{-1}=\DT{\psi(c)}$ for the conjugation of the right-handed Dehn twist $t_c$ along $c$ by the mapping class $\psi\in\Gamma_0^6$.

We first record the following generalized lantern relation. Note that the only interesting relations that may occur on a planar surface are lantern generalizations.

\begin{lemma}[A generalized lantern relation]\label{lem:key-relation}
In $\Gamma_0^6$,
\begin{equation}\label{eq:key-relation}
 \DT{u}\DT{v}\DT{w}\DT{z}\DT{\alpha}\DT{\beta}
 =
 \DT{\delta_0}\DT{\delta_1}^{2}\DT{\delta_2}^{2}
 \DT{\delta_3}^{2}\DT{\delta_4}^{2}\DT{\delta_5}.
\end{equation}
\end{lemma}

\begin{proof}
We give two proofs, the first of which reveals how we discovered this relation.

\smallskip
\noindent\emph{First proof via line arrangements.}
Consider the five real lines in $\R^2$
\begin{align*}
 L_1&:\ 8x-3y+7=0, &
 L_2&:\ 3x-2y+7=0,\\
 L_3&:\ x-7y-19=0, &
 L_4&:\ x+8y+11=0,\\
 L_5&:\ 2x+y+7=0.
\end{align*}
The arrangement is shown in Figure~\ref{fig:arrangement}.

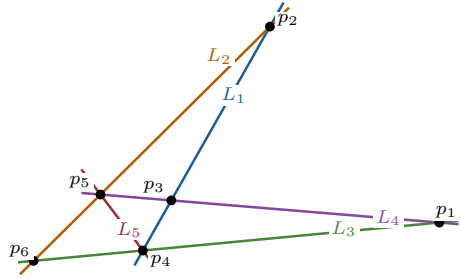
\begin{figure}[H]
\centering
\begin{tikzpicture}[x=.56cm,y=.37cm]
  \begin{scope}
  \clip (-5.05,-3.90) rectangle (5.55,5.85);
  \draw[lineone,line width=.95pt] (-2.20,-3.533)--(1.32,5.853);
  \draw[linetwo,line width=.95pt] (-4.90,-3.850)--(1.45,5.675);
  \draw[linethree,line width=.95pt] (-4.95,-3.421)--(5.45,-1.936);
  \draw[linefour,line width=.95pt] (-3.45,-.944)--(5.45,-2.056);
  \draw[linefive,line width=.95pt] (-3.45,-.100)--(-1.70,-3.600);
  \end{scope}
  \node[arrpoint] at (5,-2) {};
  \node[font=\scriptsize,fill=white,inner sep=1pt] at (5.17,-1.63) {$p_1$};
  \node[arrpoint] at (1,5) {};
  \node[font=\scriptsize,fill=white,inner sep=1pt] at (1.43,5.23) {$p_2$};
  \node[arrpoint] at (-1.328,-1.209) {};
  \node[font=\scriptsize,fill=white,inner sep=1pt] at (-1.73,-.72) {$p_3$};
  \node[arrpoint] at (-2,-3) {};
  \node[font=\scriptsize,fill=white,inner sep=1pt] at (-1.57,-3.38) {$p_4$};
  \node[arrpoint] at (-3,-1) {};
  \node[font=\scriptsize,fill=white,inner sep=1pt] at (-3.48,-.55) {$p_5$};
  \node[arrpoint] at (-4.579,-3.368) {};
  \node[font=\scriptsize,fill=white,inner sep=1pt] at (-4.90,-3.02) {$p_6$};
  \node[font=\scriptsize,text=lineone,fill=white,inner sep=1pt]
    at (.15,2.55) {$L_1$};
  \node[font=\scriptsize,text=linetwo,fill=white,inner sep=1pt]
    at (-.20,3.95) {$L_2$};
  \node[font=\scriptsize,text=linethree,fill=white,inner sep=1pt]
    at (2.75,-2.22) {$L_3$};
  \node[font=\scriptsize,text=linefour,fill=white,inner sep=1pt]
    at (3.80,-1.82) {$L_4$};
  \node[font=\scriptsize,text=linefive,fill=white,inner sep=1pt]
    at (-2.32,-2.25) {$L_5$};
\end{tikzpicture}
\caption{The line arrangement.  Hironaka's positive factors are
ordered from the leftmost point to the rightmost point.}
\label{fig:arrangement}
\end{figure}

We spell out the curve convention.  Choose a vertical fiber
$x=x_0$ to the right of all the intersection points.  Its intersections
with $L_1,\ldots,L_5$, ordered from top to bottom, initially occur in the
order
\[
  1\,2\,3\,4\,5.
\]
Move the vertical fiber from right to left, so that it passes through
$p_1,\ldots,p_6$ in this order.  Upon crossing $p_i$, the lines through
$p_i$ form a consecutive block in the current vertical ordering, and
that block is reversed.  Here the entries in $(i,j)$ refer to the
\emph{current positions} of the strands, not to the labels of the lines.
The successive orders, together with the incidence data and the associated
curves, are
\begin{center}
\renewcommand{\arraystretch}{1.18}
\setlength{\tabcolsep}{6pt}
\begin{tabular*}{\textwidth}{@{\extracolsep{\fill}}c|c|c|c|c|c@{}}
$p_i$ & \text{curve} & \text{lines through }$p_i$
& \text{order before crossing}
& \text{reversed block}
& \text{order after crossing}\\ \hline
$p_1$&$\beta$ &$L_3,L_4$       &$12345$ &$(3,4)$&$12435$\\
$p_2$&$\alpha$&$L_1,L_2$       &$12435$ &$(1,2)$&$21435$\\
$p_3$&$z$     &$L_1,L_4$       &$21435$ &$(2,3)$&$24135$\\
$p_4$&$w$     &$L_1,L_3,L_5$   &$24135$ &$(3,5)$&$24531$\\
$p_5$&$v$     &$L_2,L_4,L_5$   &$24531$ &$(1,3)$&$54231$\\
$p_6$&$u$     &$L_2,L_3$       &$54231$ &$(3,4)$&$54321$
\end{tabular*}
\end{center}
Thus the sequence of reversed position-blocks is
\begin{equation}\label{eq:block-sequence}
 (3,4),\ (1,2),\ (2,3),\ (3,5),\ (1,3),\ (3,4).
\end{equation}
This sweep is used only to compute the transport braids from the
reference fiber on the right.  The numbers of intersection points on
$L_1,L_2,L_3,L_4,L_5$ are $3,3,3,3,2$, respectively.  Hironaka's
line-arrangement relation \cite[Theorem~1.2]{Hironaka2012}, whose interior
factors occur in the reverse order $p_6,\ldots,p_1$, gives the mapping class relation
\eqref{eq:key-relation}.

\smallskip
\noindent\emph{Second proof via elementary relations.}
Let $\gamma$ be the curve in Figure~\ref{fig:curves}(j), separating
$\delta_1,\delta_2,\delta_3,\delta_4$ from $\delta_0,\delta_5$.
Choose the disjoint curves $a_{24}$ and $a_{13}$ shown in
Figure~\ref{fig:curves}(i), cutting off, respectively,
$\{\delta_2,\delta_4\}$ and $\{\delta_1,\delta_3\}$.  The curves
$\lambda$ and $\mu$ in Figure~\ref{fig:curves}(k)--(l) cut off,
respectively, $\{\delta_1,a_{24}\}$ and $\{\delta_3,a_{24}\}$ inside
the component bounded by $\gamma$.
The named curves determine four-holed subsurfaces on which the ordinary
lantern relation gives
\begin{align}
 \DT{\gamma}\DT{a_{24}}\DT{\delta_1}\DT{\delta_3}
   &=\DT{\lambda}\DT{\mu}\DT{a_{13}},
   \label{eq:lantern-one}\\
 \DT{\lambda}\DT{\delta_1}\DT{\delta_2}\DT{\delta_4}
   &=\DT{a_{24}}\DT{z}\DT{\alpha},
   \label{eq:lantern-two}\\
 \DT{\mu}\DT{\delta_2}\DT{\delta_3}\DT{\delta_4}
   &=\DT{\beta}\DT{u}\DT{a_{24}}.
   \label{eq:lantern-three}
\end{align}
Multiplying \eqref{eq:lantern-one} on the left by
\[
 \DT{a_{24}}^{-1}
 \DT{\delta_1}\DT{\delta_2}^{2}
 \DT{\delta_3}\DT{\delta_4}^{2},
\]
and using \eqref{eq:lantern-two} and \eqref{eq:lantern-three}, and repeatedly using that every $t_{\delta_i}$ is central in $\Gamma_0^6$, gives
\begin{align*}
 \DT{\gamma}\prod_{i=1}^{4}\DT{\delta_i}^{2}
 &=
 \DT{a_{24}}^{-1}
 \bigl(\DT{\lambda}\DT{\delta_1}\DT{\delta_2}\DT{\delta_4}\bigr)
 \bigl(\DT{\mu}\DT{\delta_2}\DT{\delta_3}\DT{\delta_4}\bigr)
 \DT{a_{13}}\\
  &=
\DT{a_{24}}^{-1}
 \bigl(\DT{a_{24}}\DT{z}\DT{\alpha}\bigr)
 \bigl(\DT{\beta}\DT{u}\DT{a_{24}}\bigr)
 \DT{a_{13}}\\
 &=
 \DT{z}\DT{\alpha}\DT{\beta}\DT{u}
 \DT{a_{24}}\DT{a_{13}}\\
 &=
 \DT{u}\DT{a_{24}}\DT{a_{13}}
 \DT{z}\DT{\alpha}\DT{\beta}.
 \tag{\(\ast\)}
\end{align*}
The last equality follows by cyclically permuting a factorization of the
boundary multitwist of the subsurface bounded by $\gamma$.  This
multitwist is central in the mapping class group of that subsurface, and
all the factors involved are supported there.

Finally, the component of
$\Sigma\setminus(a_{24}\cup a_{13})$ containing
$\delta_0$ and $\delta_5$ is a four-holed sphere with boundary components
$\delta_0,\delta_5,a_{24},a_{13}$, and its lantern relation is
\begin{equation}\label{eq:lantern-four}
 \DT{\delta_0}\DT{\delta_5}\DT{a_{24}}\DT{a_{13}}
 =
 \DT{\gamma}\DT{v}\DT{w}.
\end{equation}
Multiplying \((\ast)\) on the left by
$\DT{\gamma}^{-1}\DT{\delta_0}\DT{\delta_5}$, and using that
$\gamma$ is disjoint from $u,a_{24},a_{13}$, we obtain
\begin{align*}
 \DT{\delta_0}\DT{\delta_5}
 \prod_{i=1}^{4}\DT{\delta_i}^{2}
 &=
 \DT{u}
 \bigl(\DT{\gamma}^{-1}\DT{\delta_0}\DT{\delta_5}
       \DT{a_{24}}\DT{a_{13}}\bigr)
 \DT{z}\DT{\alpha}\DT{\beta}\\
 &=
 \DT{u}\DT{v}\DT{w}\DT{z}\DT{\alpha}\DT{\beta},
\end{align*}
where the second equality is \eqref{eq:lantern-four}.
\end{proof}

\begin{figure}[H]
\centering
\subfigure[$u$]{%
\begin{tikzpicture}[x=.63cm,y=.63cm]
  \PageSetup \CurveU \DrawHoles
\end{tikzpicture}}
\hspace{.012\textwidth}
\subfigure[$v$]{%
\begin{tikzpicture}[x=.63cm,y=.63cm]
  \PageSetup \CurveV \DrawHoles
\end{tikzpicture}}
\hspace{.012\textwidth}
\subfigure[$w$]{%
\begin{tikzpicture}[x=.63cm,y=.63cm]
  \PageSetup \CurveW \DrawHoles
\end{tikzpicture}}
\hspace{.012\textwidth}
\subfigure[$z$]{%
\begin{tikzpicture}[x=.63cm,y=.63cm]
  \PageSetup \CurveZ \DrawHoles
\end{tikzpicture}}

\smallskip
\subfigure[$\alpha$]{%
\begin{tikzpicture}[x=.63cm,y=.63cm]
  \PageSetup \CurveAlpha \DrawHoles
\end{tikzpicture}}
\hspace{.012\textwidth}
\subfigure[$\beta$]{%
\begin{tikzpicture}[x=.63cm,y=.63cm]
  \PageSetup \CurveBeta \DrawHoles
\end{tikzpicture}}
\hspace{.012\textwidth}
\subfigure[$d$]{%
\begin{tikzpicture}[x=.63cm,y=.63cm]
  \PageSetup \CurveD \DrawHoles
\end{tikzpicture}}
\hspace{.012\textwidth}
\subfigure[$\pi_1$ basis]{%
\begin{tikzpicture}[x=.63cm,y=.63cm]
  \PageSetup \DrawHoles \DrawBasis
\end{tikzpicture}}

\smallskip
\subfigure[$a_{13},a_{24}$]{%
\begin{tikzpicture}[x=.63cm,y=.63cm]
  \PageSetup \CurveAThirteen \CurveATwentyFour \DrawHoles
\end{tikzpicture}}
\hspace{.012\textwidth}
\subfigure[$\gamma$]{%
\begin{tikzpicture}[x=.63cm,y=.63cm]
  \PageSetup \CurveGamma \DrawHoles
\end{tikzpicture}}
\hspace{.012\textwidth}
\subfigure[$\lambda$]{%
\begin{tikzpicture}[x=.63cm,y=.63cm]
  \PageSetup \CurveLambda \DrawHoles
\end{tikzpicture}}
\hspace{.012\textwidth}
\subfigure[$\mu$]{%
\begin{tikzpicture}[x=.63cm,y=.63cm]
  \PageSetup \CurveMu \DrawHoles
\end{tikzpicture}}
\caption{The curves on $\Sigma$.  The first row is the Dehn twist curves in
$B=\DT{u}\DT{v}\DT{w}\DT{z}$.  The second row contains the conjugating
curve $\alpha$, the complementary curve $\beta$, the fixed vanishing cycle
$d$, and the oriented curves $x_1,\ldots,x_5$ generating $\pi_1(\Sigma, *)$.  The third row
shows the auxiliary curves used in the second proof of
Lemma~\ref{lem:key-relation}.}
\label{fig:curves}
\end{figure}
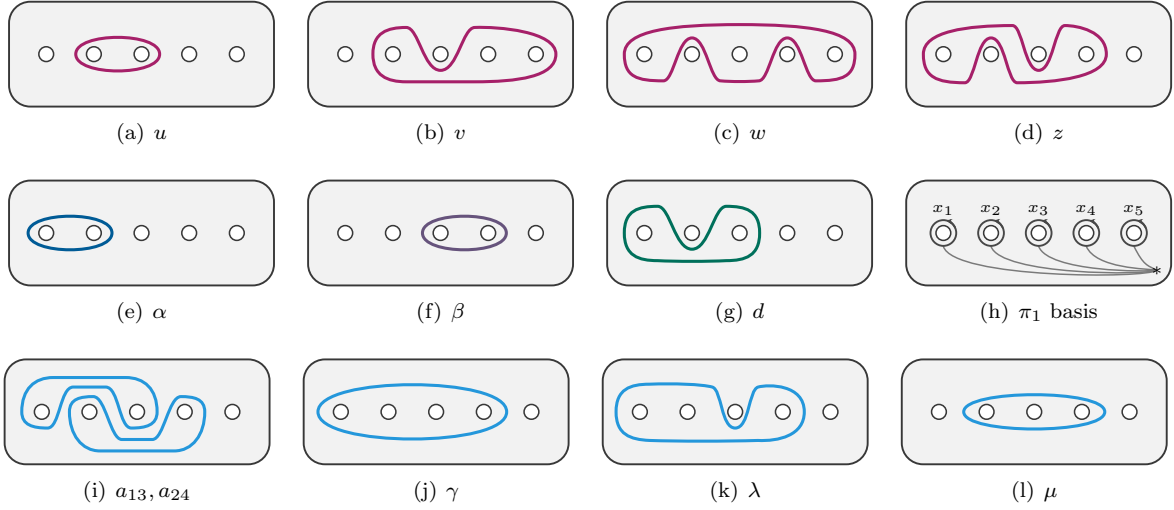

Set
\[
 B=\DT{u}\DT{v}\DT{w}\DT{z},\qquad
 \Delta=\DT{\delta_0}\DT{\delta_1}^{2}\DT{\delta_2}^{2}
 \DT{\delta_3}^{2}\DT{\delta_4}^{2}\DT{\delta_5}.
\]
Then Lemma~\ref{lem:key-relation} reads
\begin{equation}\label{eq:BT}
 B\DT{\alpha}\DT{\beta}=\Delta.
\end{equation}

\begin{proof}[Proof of Theorem~\ref{thm:main}]
Since $\Delta$ is central, \eqref{eq:BT} implies
$[B,\DT{\alpha}\DT{\beta}]=1$.  Conjugating the multitwist by $B$ gives
\[
 \DT{B(\alpha)}\DT{B(\beta)}=\DT{\alpha}\DT{\beta}.
\]
By uniqueness of a multitwist decomposition,
$\{B(\alpha),B(\beta)\}=\{\alpha,\beta\}$.  Every element of
$\Gamma_0^6$ fixes each boundary component, so $B$ cannot
interchange $\alpha$, which cuts off $\delta_1,\delta_2$, with $\beta$,
which cuts off $\delta_3,\delta_4$.  Hence $B(\alpha)=\alpha$ and
$B(\beta)=\beta$, and in particular
\begin{equation}\label{eq:B-commutes}
 [B,\DT{\alpha}]=1.
\end{equation}

For $c\in\{u,v,w,z\}$ and $n\geq0$, put
\[
 c_n=t_\alpha^n(c).
\]
Define
\begin{equation}\label{eq:Pn}
 P_n:=\DT{d}\DT{u_n}\DT{v_n}\DT{w_n}\DT{z_n}.
\end{equation}
By \eqref{eq:B-commutes}, every $P_n$ is a positive factorization of the
fixed mapping class
\begin{equation}\label{eq:fixed-monodromy}
 \phi=\DT{d}B=\DT{d}\DT{u}\DT{v}\DT{w}\DT{z}.
\end{equation}
Every curve in Figure~\ref{fig:curves} separates at least two boundary
components from at least two others.  Thus none is null-homotopic or
boundary-parallel, and the same holds for $u_n,v_n,w_n,z_n$.
This is not the global conjugation of $P_0$ by $t_\alpha^n$:
the minimal intersection number $i(d,\alpha)>0$, while the factor $\DT{d}$ is fixed in \eqref{eq:Pn}.

For elements $g_1,\ldots,g_k$ of a group, let
$N(g_1,\ldots,g_k)$ denote their normal closure.  Let $X_n\to D^2$ be the
allowable Lefschetz fibration prescribed by $P_n$; its total space is Stein
\cite{LoiPiergallini2001,AkbulutOzbagci2001}.  The vanishing-cycle presentation gives
\begin{equation}\label{eq:filling-group}
 \pi_1(X_n)\cong
 \pi_1(\Sigma)\big/N(d,u_n,v_n,w_n,z_n).
\end{equation}
Choose the basepoint and the oriented meridians in
Figure~\ref{fig:curves}(h), where each $x_i$ is oriented counterclockwise
around $\delta_i$, and write
\[
 \pi_1(\Sigma)=F_5=\langle x_1,x_2,x_3,x_4,x_5\rangle.
\]
We read path products from left to right and orient each of
$u,v,w,z,d$ as the positive boundary of the component of its complement
not containing $\delta_0$.  Changing a representative by conjugation or
inversion does not alter its normal closure.  With these conventions, we may
take
\begin{align}
 u&=x_3x_2, &
 v&=x_5x_4x_2,\label{eq:curve-words-1}\\
 w&=x_5x_4x_3x_2x_1x_2^{-1}x_4^{-1}, &
 z&=x_4x_2x_1x_2^{-1},\label{eq:curve-words-2}\\
 d&=x_3x_1.&&\label{eq:curve-words-3}
\end{align}
First quotient by the four vanishing cycles in $B$:
\[
 Q=F_5/N(u,v,w,z).
\]
Set $a=x_1$ and $b=x_2$.  The relations $u=1$, $z=1$, and $v=1$
give successively
\[
 x_3=b^{-1},\qquad
 x_4=ba^{-1}b^{-1},\qquad
 x_5=ab^{-1}.
\]
After these substitutions,
\[
 w=b^{-1}a^2b^{-1},
\]
so the remaining relation $w=1$ is equivalent to $a^2=b^2$.  Hence
\begin{equation}\label{eq:Qab}
 Q\cong\langle a,b\mid a^2=b^2\rangle.
\end{equation}
Setting $r=ab^{-1}$, so that $a=rb$, rewrites
\eqref{eq:Qab} as
\begin{equation}\label{eq:QKlein}
 Q\cong\langle r,b\mid brb^{-1}=r^{-1}\rangle,
\end{equation}
the Klein-bottle group.

Apply the automorphism induced by $t_\alpha^{-n}$ to
\eqref{eq:filling-group}.  We obtain
\begin{equation}\label{eq:move-conjugation}
 \pi_1(X_n)\cong
 Q/N\bigl(\overline{t_\alpha^{-n}(d)}\bigr),
\end{equation}
where the bar denotes the equivalence class of $t_\alpha^{-n}(d)$ in the quotient $Q$. The conjugating curve $\alpha$ is represented by
\[
 p=x_2x_1.
\]
With the orientation and basepoint conventions above,
\[
 t_\alpha^{-n}(x_i)=p^{-n}x_ip^n\qquad(i=1,2),
\]
while $x_3,x_4,x_5$ are fixed.  Since $d=x_3x_1$,
\begin{equation}\label{eq:moving-relator}
 t_\alpha^{-n}(d)=x_3p^{-n}x_1p^n.
\end{equation}
In $Q$, we have $x_3=b^{-1}$.  If $q=a^2=b^2$, then $q$ is
central.  Moreover, since $a=rb$ and $brb^{-1}=r^{-1}$,
\[
 p=ba=r^{-1}q.
\]
Therefore
\begin{align*}
 \overline{t_\alpha^{-n}(d)}
 &=b^{-1}p^{-n}ap^n\\
 &=b^{-1}(r^{-1}q)^{-n}a(r^{-1}q)^n\\
 &=b^{-1}q^{-n}r^na r^{-n}q^n\\
 &=b^{-1}r^na r^{-n}\\
 &=b^{-1}r^{n+1}br^{-n}\\
 &=r^{-(n+1)}r^{-n}\\
 &=r^{-(2n+1)}.
\end{align*}
Combining this with \eqref{eq:move-conjugation} and
\eqref{eq:QKlein}, we obtain
\begin{equation}\label{eq:final-group}
 \pi_1(X_n)\cong
 \left\langle r,b\mid r^{2n+1}=1,\;brb^{-1}=r^{-1}\right\rangle
 \cong \Z/(2n+1)\rtimes_{-1}\Z.
\end{equation}
Here the generator $b$ of the infinite cyclic factor acts on
$\langle r\rangle\cong\Z/(2n+1)$ by inversion.  The element $r$ has
exact order $2n+1$.  Moreover, every finite order element lies in
$\langle r\rangle$, since its image under the projection to the infinite
cyclic factor must be trivial.  Thus the maximal order of a finite order
element is $2n+1$. It follows that the groups $\pi_1(X_n)$ are pairwise
nonisomorphic, and consequently the Stein fillings $X_n$ are pairwise
non-homeomorphic.

Abelianizing \eqref{eq:final-group} gives $2r=0$ and $(2n+1)r=0$, hence
$r=0$ and
\[
 H_1(X_n;\Z)\cong\Z.
\]
The page has Euler characteristic $-4$ and there are five critical points,
so $\chi(X_n)=1$.  An allowable Lefschetz fibration over the disk has a
handle decomposition with only $0$--, $1$--, and $2$--handles.  It follows
that $H_k(X_n;\Z)=0$ for $k\geq3$ and that $H_2(X_n;\Z)$ is free.  Since
$X_n$ is connected and $H_1(X_n;\Z)\cong\Z$, the Euler characteristic
calculation gives
\[
 H_0(X_n;\Z)=H_1(X_n;\Z)=H_2(X_n;\Z)=\Z.
\]

This proves the theorem for the planar contact manifold $(Y,\xi)$ supported
by $(\Sigma,\phi)$.  For $m\geq0$, form
\[
 (Y_m,\xi_m)=(Y,\xi)\#^m(S^1\times S^2,\xi_{\rm std}).
\]
The contact connected sum is supported by the planar Murasugi sum, and Eliashberg's connected-sum theorem
identifies its Stein fillings with boundary connected sums of Stein fillings
of the summands \cite{Eliashberg1991}.  In particular,
\[
 X_{n,m}=X_n\,\natural\,\bigl(\natural^m(S^1\times D^3)\bigr),
 \qquad
 \pi_1(X_{n,m})\cong\pi_1(X_n)*F_m.
\]
The manifolds $Y_m$ are pairwise non-homeomorphic since their first Betti
numbers are distinct.  Every finite-order element of a free product is
conjugate into a factor, so the largest finite order in $\pi_1(X_{n,m})$ is
$2n+1$.  Thus, for each $m$, the fillings $X_{n,m}$ are pairwise
non-homeomorphic.  Of course, one may more generally take a planar Murasugi
sum with any fixed planar contact $3$--manifold admitting a Stein filling; the corresponding boundary connected sums again yield
infinitely many pairwise non-homeomorphic Stein fillings.
\end{proof}

\begin{remark}
The fillings $X_n$ are related by fibered Luttinger surgeries \cite{baykur-inequivalent}.  Indeed,
choose a simple closed curve $\gamma\subset D^2$ enclosing precisely the
four critical values whose vanishing cycles are $u_n,v_n,w_n,z_n$.
With the chosen distinguished paths, the monodromy around $\gamma$ is
$B$.  Let $S_\alpha$ denote the component of
$\Sigma\setminus\alpha$ containing $\delta_1$ and $\delta_2$.
Since $B(\alpha)=\alpha$ and $B$ fixes every boundary component,
$B$ preserves $S_\alpha$, and therefore preserves the orientation of
$\alpha=\partial S_\alpha$.  It follows that the parallel transport  of $\alpha$ over $\gamma$ is a fibered Lagrangian torus
$T_\alpha$. Passing from $P_n$ to $P_{n+1}$ is the partial conjugation of the four-factor subword by $t_\alpha$, and corresponds to a
Luttinger surgery along $T_\alpha$ with $\alpha$ as the surgery curve.

It is natural to ask whether, for every planar contact $3$--manifold, its Stein
fillings form only finitely many equivalence classes under the relation generated
by such surgeries.
\end{remark}

\medskip
\noindent \textit{Acknowledgements.}
The author was partially supported by NSF grant 
DMS-2506431 and a Simons Foundation Travel Grant.

\bigskip
\bibliography{refs}
\bibliographystyle{plain}

\end{document}